\documentclass[12pt,reqno]{amsart}
\usepackage{amsthm, amscd, amsfonts, amssymb, graphicx, color}
\usepackage[bookmarksnumbered, colorlinks, plainpages,linkcolor=green,anchorcolor=blue,citecolor=blue,urlcolor=blue]{hyperref}
\usepackage{enumerate}
\usepackage{exscale}
\usepackage{relsize}
\usepackage{pdfsync}
\numberwithin{equation}{section}

\makeatletter
\@namedef{subjclassname@2020}{%
  \textup{2020} Mathematics Subject Classification}
\makeatother

\usepackage[includemp,body={398pt,550pt},footskip=30pt,%
            marginparwidth=60pt,marginparsep=10pt]{geometry}

\newtheorem{theorem}{Theorem}[section]
\newtheorem{lemma}[theorem]{Lemma}
\newtheorem{proposition}[theorem]{Proposition}
\newtheorem{corollary}[theorem]{Corollary}
\theoremstyle{definition}
\newtheorem{definition}[theorem]{Definition}
\newtheorem{example}[theorem]{Example}

\newtheorem{remark}[theorem]{Remark}
\numberwithin{equation}{section}

\newtheorem{theory}{Theorem}

\newcommand{\A}{\mathcal{A}}

\newcommand{\bb}{\mathbb{B}}

\newcommand{\C}{\mathbb{C}}

\newcommand{\ct}{\mathcal{T}}
\newcommand{\Z}{\mathbb{Z}}
\newcommand{\res}{\mathrm{restricted}}

\begin{document}
\title[Weighted theory of Toeplitz operators on the Fock spaces]{Weighted theory of Toeplitz operators on the Fock spaces}
\date{May 15, 2025.
}
\author[J. Chen]{Jiale Chen}
\address{School of Mathematics and Statistics, Shaanxi Normal University, Xi'an 710119, China.}
\email{jialechen@snnu.edu.cn}

\thanks{
}



\begin{abstract}
  \noindent We study the weighted compactness and boundedness of Toeplitz operators on the Fock spaces. Fix $\alpha>0$. Let $T_{\varphi}$ be the Toeplitz operator on the Fock space $F^2_{\alpha}$ over $\mathbb{C}^n$ with symbol $\varphi\in L^{\infty}$. For $1<p<\infty$ and any finite sum $T$ of finite products of Toeplitz operators $T_{\varphi}$'s, we show that $T$ is compact on the weighted Fock space $F^p_{\alpha,w}$ if and only if its Berezin transform vanishes at infinity, where $w$ is a restricted $A_p$-weight on $\mathbb{C}^n$. Concerning boundedness, for $1\leq p<\infty$, we characterize the $r$-doubling weights $w$ such that $T_{\varphi}$ is bounded on the weighted spaces $L^p_{\alpha,w}$ via a $\varphi$-adapted $A_p$-type condition. Our method also establishes a two weight inequality for the Fock projections in the case of $r$-doubling weights. Moreover, we characterize the corresponding weighted compactness of Bergman--Toeplitz operators, which answers a question raised by Stockdale and Wagner [Math. Z. 305 (2023), no. 1, Paper No. 10].
\end{abstract}
\maketitle


\section{Introduction}
\allowdisplaybreaks[4]
The purpose of this paper is to investigate the weighted properties of Toeplitz operators on the Fock spaces over $\C^n$. We first recall some basic notions.  A nonnegative function $w$ on $\C^n$ is said to be a weight if it is locally integrable on $\C^n$. Given $\alpha>0$, $1\leq p<\infty$ and a weight $w$ on $\C^n$, the weighted space $L^p_{\alpha,w}$ consists of measurable functions $f$ on $\C^n$ such that
$$\|f\|^p_{L^p_{\alpha,w}}:=\int_{\C^n}|f(z)|^pe^{-\frac{p\alpha}{2}|z|^2}w(z)dv(z)<\infty,$$
where $dv$ is the Lebesgue measure on $\C^n$. Let $\mathcal{H}(\C^n)$ be the space of entire functions on $\C^n$. Then the weighted Fock space $F^p_{\alpha,w}$ is defined by
$$F^p_{\alpha,w}:=L^p_{\alpha,w}\cap\mathcal{H}(\C^n)$$
with the inherited norm. If $w\equiv\left(\frac{p\alpha}{2\pi}\right)^n$, then the above spaces are simply denoted by $L^p_{\alpha}$ and $F^p_{\alpha}$. We refer to \cite{JPR,Zh} for more information of the classical Fock spaces $F^p_{\alpha}$. Since $F^2_{\alpha}$ is a closed subspace of $L^2_{\alpha}$, there exists an orthogonal projection from $L^2_{\alpha}$ onto $F^2_{\alpha}$. This map is the so-called Fock projection and is given by
$$P_{\alpha}(f)(z):=\int_{\C^n}f(u)\overline{K_z(u)}d\lambda_{\alpha}(u),$$
where $K_z(u)=e^{\alpha\langle u,z\rangle}$ is the reproducing kernel of $F^2_{\alpha}$, and
$$d\lambda_{\alpha}(u):=\left(\frac{\alpha}{\pi}\right)^ne^{-\alpha|u|^2}dv(u)$$
is the Gaussian measure. It is well-known that for any $1\leq p<\infty$, $P_{\alpha}$ actually extends to a bounded projection from $L^p_{\alpha}$ onto $F^p_{\alpha}$; see \cite[Theorem 7.1]{JPR}.

Given a function $\varphi$ on $\C^n$, the associated Toeplitz operator $T_{\varphi}$ is formally defined by
$$T_{\varphi}f:=P_{\alpha}(\varphi f).$$
By the mapping properties of the Fock projection $P_{\alpha}$, it is clear that if $\varphi\in L^{\infty}:=L^{\infty}(\C^n)$, then $T_{\varphi}$ extents to a bounded operator from $L^p_{\alpha}$ into $F^p_{\alpha}$ for $1\leq p<\infty$. The properties of Toeplitz operators acting on classical Fock spaces have been investigated extensively; see \cite{BCK,BCI,BvSW,CIL,Fu,HL,IZ} and the references therein. A main theme in the study of Toeplitz operators is to characterize the compactness of operators from the Toeplitz algebra via the Berezin transform; see for instance \cite{AZ,En,JZ,MSW,Su,WX} for the Bergman space case and \cite{BI,Fu2,HLW,Is15,IMW,XZ} for the Fock space case. Recall that for a bounded linear operator $T$ on the Fock space $F^p_{\alpha}$, its Berezin transform is defined by
$$\widetilde{T}(z):=\langle Tk_z,k_z\rangle_{\alpha},\quad z\in\C^n,$$
where $k_z(u)=e^{\alpha\langle u,z\rangle-\frac{\alpha}{2}|z|^2}$ is the normalized reproducing kernel of $F^2_{\alpha}$, and the integral pairing $\langle\cdot,\cdot\rangle_{\alpha}$ is given by
$$\langle f,g\rangle_{\alpha}:=\int_{\C^n}f(z)\overline{g(z)}d\lambda_{\alpha}(z).$$
The following theorem, proved by Bauer and Isralowitz \cite[Theorem 1.1]{BI}, characterizes the compactness of operators in the Toeplitz algebra on $F^p_{\alpha}$ via the Berezin transform.

\begin{theory}[\cite{BI}]\label{bere}
{\it Let $\alpha>0$ and $1<p<\infty$. Suppose that $T$ is an operator in the norm closure of the algebra generated by Toeplitz operators with
$L^{\infty}$-symbols acting on $F^p_{\alpha}$. Then $T$ is compact on $F^p_{\alpha}$ if and only if $\lim_{|z|\to\infty}\widetilde{T}(z)=0$.}
\end{theory}

Recently, the weighted boundedness of the Fock projections has drawn some attention; see \cite{CFP,CW24,CW24-1,Is14}. Let $Q$ denote a cube in $\C^n$, and write $l(Q)$ for its side length. As usual, $p'$ denotes the conjugate exponent of $p$, i.e. $1/p+1/p'=1$. For $1<p<\infty$, we say a weight $w$ belongs to the class $A^{\mathrm{restricted}}_p$ if for some (or any) fixed $r>0$,
$$[w]_{A_{p,r}}:=\sup_{Q:l(Q)=r}\left(\frac{1}{v(Q)}\int_{Q}wdv\right)
\left(\frac{1}{v(Q)}\int_{Q}w^{-\frac{p'}{p}}dv\right)^{\frac{p}{p'}}<\infty,$$
and $w$ is said to be in the class $A^{\res}_1$ if for some (or any) fixed $r>0$,
$$[w]_{A_{1,r}}:=\sup_{Q:l(Q)=r}\left(\frac{1}{v(Q)}\int_Qwdv\right)\|w^{-1}\|_{L^{\infty}(Q)}<\infty.$$
The following theorem, characterizing the boundedness of $P_{\alpha}$ on $L^p_{\alpha,w}$, was proved by Isralowitz \cite[Theorem 3.1]{Is14} for the case $1<p<\infty$ and by Cascante--F\`{a}brega--Pel\'{a}ez \cite[Proposition 2.7]{CFP} for the case $p=1$.

\begin{theory}[\cite{CFP,Is14}]\label{one-weight}
{\it Let $\alpha>0$, $1\leq p<\infty$, and let $w$ be a weight on $\C^n$. Then the Fock projection $P_{\alpha}$ is bounded on $L^p_{\alpha,w}$ if and only if $w\in A^{\res}_p$.}
\end{theory}

The aim of this paper is twofold: we first generalize Theorem \ref{bere} to the weighted Fock spaces $F^p_{\alpha,w}$ induced by $w\in A^{\res}_p$, and then establish the version of Theorem \ref{one-weight} for Toeplitz operators.

We now state the main results of this paper. By Theorem \ref{one-weight}, the boundedness of Toeplitz operators with $L^{\infty}$-symbols on the weighted Fock spaces $F^p_{\alpha,w}$ induced by $w\in A^{\res}_p$ is out of question. Our first result characterizes the weighted compactness of Toeplitz operators via their Berezin transforms, which completely answers the Fock space analogue of \cite[Open Question 1.14]{SW}.

\begin{theorem}\label{main1}
Let $\alpha>0$, $1<p<\infty$, $w\in A^{\res}_p$, and let $T$ be a finite sum of finite products of Toeplitz operators with $L^{\infty}$-symbols. Then $T$ is compact on $F^p_{\alpha,w}$ if and only if $\lim_{|z|\to\infty}\widetilde{T}(z)=0$.
\end{theorem}

The proof of Theorem \ref{main1} is based on a Riesz--Kolmogorov type characterization of precompact sets in $F^p_{\alpha,w}$ and the notion of weakly localized operators on $F^p_{\alpha,w}$, which is adapted from \cite{IMW}. We will show that each Toeplitz operator with $L^{\infty}$-symbol is weakly localized on $F^{p}_{\alpha,w}$, and all weakly localized operators on $F^p_{\alpha,w}$ form an algebra. This, combined with the Riesz--Kolmogorov type compactness criterion, will give the proof of Theorem \ref{main1}.

Our second result is a characterization of weighted boundedness of Toeplitz operators under a mild assumption on weights. For $r>0$, a weight $w$ on $\C^n$ is said to be $r$-doubling if there exists some constant $C=C(r)\geq1$ such that $w(Q_{2r}(z))\leq Cw(Q_{r}(z))$ for any $z\in\C^n$. Here and in the sequel, $Q_r(z)$ denotes the cube centered at $z\in\C^n$ with side length $r>0$, and $w(E)=\int_Ewdv$ for measurable sets $E$. It is easy to see that if $w$ is $r$-doubling, then for each $m\geq1$, there exists $C=C(m,r)\geq1$ such that $w(Q_{mr}(z))\leq Cw(Q_r(z))$ holds for any $z\in\C^n$. We characterize the $r$-doubling weights $w$ such that $T_{\varphi}$ is bounded on $L^p_{\alpha,w}$ via a $\varphi$-adapted $A_{p,r}$-condition.

\begin{theorem}\label{main2}
Let $\alpha>0$, $\varphi\in L^{\infty}$, and let $w$ be an $r$-doubling weight on $\C^n$ for some $r>0$.
\begin{enumerate}
	\item [(1)] For $1<p<\infty$, $T_{\varphi}$ is bounded on $L^p_{\alpha,w}$ if and only if
	$$\sup_{Q:l(Q)=r}\left(\frac{1}{v(Q)}\int_{Q}wdv\right)
	\left(\frac{1}{v(Q)}\int_{Q}|\varphi|^{p'}w^{-\frac{p'}{p}}dv\right)^{\frac{p}{p'}}<\infty.$$
	\item [(2)] $T_{\varphi}$ is bounded on $L^1_{\alpha,w}$ if and only if
	$$\sup_{Q:l(Q)=r}\left(\frac{1}{v(Q)}\int_{Q}wdv\right)\|\varphi w^{-1}\|_{L^{\infty}(Q)}<\infty.$$
\end{enumerate}
\end{theorem}

To prove the necessary parts of the above theorem, we will consider a class of rank one operators induced by the normalized reproducing kernels of $F^2_{\alpha}$. The sufficient parts are based on some elementary estimates of $r$-doubling weights.

Using the same method as in Theorem \ref{main2}, we can also establish a two weight inequality for the Fock projections. After the pioneering works of Sawyer \cite{Sa82,Sa88}, the two weight inequalities for classical operators, such as the Hilbert transform and the Bergman projection, have become a main subject in harmonic analysis and operator theory; see for instance \cite{APR,Br,FW,La14,LSSU,LSU} and the references therein. Generally, in order to obtain the two weight boundedness, one need to consider the Sawyer-type testing, which is usually hard to verify. We here characterize the two weight boundedness of the Fock projections only by some joint $A_p$-conditions, under the mild assumption that the weight of the target space is $r$-doubling. Let $r>0$ and $w,\sigma$ be two weights on $\C^n$. For $1<p<\infty$, the joint $A_{p,r}$-characteristic of $w$ and $\sigma$ is defined by
$$[w,\sigma]_{A_{p,r}}:=\sup_{Q:l(Q)=r}\left(\frac{1}{v(Q)}\int_{Q}wdv\right)
\left(\frac{1}{v(Q)}\int_{Q}\sigma^{-\frac{p'}{p}}dv\right)^{\frac{p}{p'}},$$
and the joint $A_{1,r}$-characteristic is defined by
$$[w,\sigma]_{A_{1,r}}:=\sup_{Q:l(Q)=r}\left(\frac{1}{v(Q)}\int_{Q}wdv\right)\|\sigma^{-1}\|_{L^{\infty}(Q)}.$$
We are now ready to state our two weight inequality for the Fock projections, which generalizes Theorem \ref{one-weight}.

\begin{theorem}\label{main3}
Let $\alpha>0$, $1\leq p<\infty$, and let $\sigma,w$ be two weights on $\C^n$. Suppose that $w$ is $r$-doubling for some $r>0$. Then $P_{\alpha}:L^p_{\alpha,\sigma}\to L^p_{\alpha,w}$ is bounded if and only if $[w,\sigma]_{A_{p,r}}<\infty$. Moreover,
$$\|P_{\alpha}\|_{L^p_{\alpha,\sigma}\to L^p_{\alpha,w}}\asymp[w,\sigma]^{1/p}_{A_{p,r}},$$
where the implicit constant depends on $p,n,\alpha,r$ and the $r$-doubling constant of $w$.
\end{theorem}

The following example indicates that the $r$-doubling condition in the above theorem cannot be removed.

\begin{example}
Let $\alpha>0$, and let the weights $\sigma,w$ be defined by
$$\sigma(z)=e^{-\alpha|z|^2},\quad w(z)=e^{-4\alpha|z|^2},\quad z\in\C^n.$$
Then it is easy to verify that for any $r>0$,
$$[w,\sigma]_{A_{2,r}}=\sup_{Q:l(Q)=r}\left(\frac{1}{v(Q)}\int_{Q}e^{-4\alpha|z|^2}dv(z)\right)
\left(\frac{1}{v(Q)}\int_{Q}e^{\alpha|z|^2}dv(z)\right)<\infty.$$
However, by \cite[Theorem 1]{BEY}, $P_{\alpha}:L^2_{\alpha,\sigma}\to L^2_{\alpha,w}$ is unbounded.
\end{example}

The rest part of this paper is organized as follows. In Section \ref{pre}, we recall some preliminary results. Theorem \ref{main1} is proved is Section \ref{wcpt}. Section \ref{wbdd} is devoted to proving Theorems \ref{main2} and \ref{main3}. In Section \ref{Bergman}, we consider the weighted compactness of Toeplitz operators on the Bergman space over the unit ball $\bb_n$ of $\C^n$. For $1<p<\infty$ and a B\'{e}koll\`{e}--Bonami $B_p$-weight $\sigma$, we show that the Bergman--Toeplitz operator $T_{\varphi}$ with bounded symbol $\varphi$ is compact on the weighted Bergman space $\A^p_{\sigma}$ if and only if its Berezin transform vanishes at the boundary of $\bb_n$, which answers \cite[Open Question 1.14]{SW}; see Theorem \ref{main4}.

{\it Notation.} Throughout the paper, we write $A\lesssim B$ to denote that $A\leq CB$ for some inessential constant $C>0$, and if $A\lesssim B\lesssim A$, then we write $A\asymp B$. For a weight $w$ on $\C^n$, the weight $\widehat{w}$ is defined by $\widehat{w}(z)=w(Q_1(z))$, $z\in\C^n$, and the dual weight $w'$ is defined by $w'=w^{-p'/p}$. For $r>0$ and $z\in\C^n$, we use $B_r(z)$ to denote the Euclidean ball in $\C^n$ centered at $z$ with radius $r$.

\section{Preliminaries}\label{pre}

In this section, we collect some preliminary results that will be used in the sequel.

We begin from the following elementary lemma concerning weights, which was established in \cite[Lemmas 3.2 and 3.4]{Is14}. Here and in the sequel, we will treat $r\Z^{2n}$ as a subset of $\C^n$ in the canonical way. 

\begin{lemma}\label{ele}
Let $w$ be a weight on $\C^n$.
\begin{enumerate}
	\item [(1)] If $w\in A^{\res}_p$ for some $1<p<\infty$, then $w$ is $r$-doubling for each $r>0$.
	\item [(2)] If $w$ is $r$-doubling for some $r>0$, then there exists $C>0$ such that for any $\nu,\nu'\in r\Z^{2n}$,
	$$\frac{w(Q_r(\nu))}{w(Q_r(\nu'))}\leq C^{|\nu-\nu'|}.$$
\end{enumerate}
\end{lemma}

For fixed $r>0$ and any $u,z\in\C^n$, there exist $\nu,\nu'\in r\Z^{2n}$ such that $u\in Q_r(\nu)$ and $z\in Q_r(\nu')$. Consequently, if $w\in A^{\res}_p$ for some $1<p<\infty$, then by Lemma \ref{ele}, there exist $C,C_1>0$ such that
\begin{align}\label{ele1}
w(Q_r(u))&\leq w(Q_{2r}(\nu))\leq C_1w(Q_r(\nu))\leq C_1C^{|\nu-\nu'|}w(Q_r(\nu'))\nonumber\\
&\leq C^2_1C^{|\nu-\nu'|}w(Q_r(z))\leq C^2_1C^{\sqrt{2n}r+|u-z|}w(Q_r(z)).
\end{align}

The following lemma can be deduced by \cite[Proposition 3.3 and Lemma 3.5]{CFP}.

\begin{lemma}\label{equal}
Let $0<\alpha<\infty$, $1<p<\infty$ and $w\in A^{\res}_p$. Then for any $f\in F^p_{\alpha,w}$ and $z\in\C^n$,
$$f(z)=P_{\alpha}f(z)=\langle f,K_z\rangle_{\alpha}.$$
\end{lemma}

We also need the following norm estimate, which was established in \cite[Proposition 4.1]{CFP}.

\begin{lemma}\label{test}
Let $0<\alpha<\infty$, $1<p<\infty$ and $w\in A^{\res}_p$. Then for any $z\in\C^n$,
$$\|K_z\|_{F^p_{\alpha,w}}\asymp e^{\frac{\alpha}{2}|z|^2}w(Q_1(z))^{1/p}.$$
\end{lemma}

Note that if $w\in A^{\res}_p$ for some $1<p<\infty$, then $w'\in A^{\res}_{p'}$. Moreover, by H\"{o}lder's inequality,
$$w(Q_1(z))^{1/p}\asymp w'(Q_1(z))^{-1/p'},\quad z\in\C^n.$$
Consequently, by Lemma \ref{test}, if $w\in A^{\res}_p$ for some $1<p<\infty$, then
\begin{equation}\label{test'}
\|K_z\|_{F^{p'}_{\alpha,w'}}\asymp e^{\frac{\alpha}{2}|z|^2}w'(Q_1(z))^{1/p'}\asymp e^{\frac{\alpha}{2}|z|^2}w(Q_1(z))^{-1/p},\quad z\in\C^n.
\end{equation}
In particular,
\begin{equation}\label{prod}
\|K_z\|_{F^p_{\alpha,w}}\cdot\|K_z\|_{F^{p'}_{\alpha,w'}}\asymp e^{\alpha|z|^2},\quad z\in\C^n.
\end{equation}

The following two lemmas were proved in \cite[Lemmas 3.1 and 3.2]{CHW}.

\begin{lemma}\label{dual}
Let $\alpha>0$, $1<p<\infty$ and $w\in A^{\res}_p$. Then the dual space of $F^p_{\alpha,w}$ can be identified with $F^{p'}_{\alpha,w'}$ under the pairing $\langle\cdot,\cdot\rangle_{\alpha}$.
\end{lemma}

\begin{lemma}\label{norm-eq}
Let $\alpha>0$, $1<p<\infty$ and $w\in A^{\res}_p$. Then $F^p_{\alpha,w}=F^p_{\alpha,\widehat{w}}$ with equivalent norms.
\end{lemma}

\section{Weighted compactness}\label{wcpt}

The purpose of this section is to prove Theorem \ref{main1}. To this end, we introduce the following definition of weakly localized operators, which is adapted from \cite{IMW}. For $z\in\C^n$, we use $k_z^{(p,w)}$ to denote the ``$(p,w)$-normalized'' reproducing kernel at $z$. That is,
$$k^{(p,w)}_z:=\frac{K_z}{\|K_z\|_{F^p_{\alpha,w}}}.$$

\begin{definition}
Let $\alpha>0$, $1<p<\infty$, and let $w$ be a weight on $\C^n$. A bounded linear operator $T$ on $F^p_{\alpha,w}$ is said to be weakly localized if the following conditions hold:
\begin{equation*}
\sup_{z\in\C^n}\int_{\C^n}|\langle Tk^{(p,w)}_z,k^{(p',w')}_u\rangle_{\alpha}|dv(u)<\infty,
\end{equation*}
\begin{equation}\label{sup2}
\sup_{u\in\C^n}\int_{\C^n}|\langle Tk^{(p,w)}_z,k^{(p',w')}_u\rangle_{\alpha}|dv(z)<\infty,
\end{equation}
\begin{equation*}
\lim_{r\to\infty}\sup_{z\in\C^n}\int_{B_r(z)^c}|\langle Tk^{(p,w)}_z,k^{(p',w')}_u\rangle_{\alpha}|dv(u)=0,
\end{equation*}
and
\begin{equation}\label{vanishing2}
\lim_{r\to\infty}\sup_{u\in\C^n}\int_{B_r(u)^c}|\langle Tk^{(p,w)}_z,k^{(p',w')}_u\rangle_{\alpha}|dv(z)=0.
\end{equation}
\end{definition}

We now establish two important properties for weakly localized operators on the Fock spaces $F^p_{\alpha,w}$ induced by $w\in A^{\res}_p$. The first one indicates that all Toeplitz operators with bounded symbols are weakly localized on $F^p_{\alpha,w}$, and the second one says that all weakly localized operators on $F^p_{\alpha,w}$ form an algebra.

\begin{proposition}\label{wl}
Let $\alpha>0$, $1<p<\infty$ and $w\in A^{\res}_p$. Then for any $\varphi\in L^{\infty}$, $T_{\varphi}$ is a weakly localized operator on $F^p_{\alpha,w}$.
\end{proposition}
\begin{proof}
We only show that $T_{\varphi}$ satisfies the conditions \eqref{sup2} and \eqref{vanishing2}. The other two are similar. To this end, we apply Lemmas \ref{equal} and \ref{test} and the estimate \eqref{test'} to obtain that for any $z,u\in\C^n$,
\begin{align*}
&|\langle T_{\varphi}k^{(p,w)}_z,k^{(p',w')}_u\rangle_{\alpha}|\\
&\ =|\langle P_{\alpha}\left(\varphi k^{(p,w)}_z\right),k^{(p',w')}_u\rangle_{\alpha}|\\
&\ =|\langle \varphi k^{(p,w)}_z,k^{(p',w')}_u\rangle_{\alpha}|\\
&\ =\left|\int_{\C^n}\varphi(\xi)k^{(p,w)}_z(\xi)\overline{k^{(p',w')}_u(\xi)}d\lambda_{\alpha}(\xi)\right|\\
&\ \leq\|\varphi\|_{L^{\infty}}\int_{\C^n}|\langle k^{(p,w)}_z,K_{\xi}\rangle_{\alpha}|
    |\langle K_{\xi},k^{(p',w')}_u\rangle_{\alpha}|d\lambda_{\alpha}(\xi)\\
&\ \asymp\|\varphi\|_{L^{\infty}}\left(\frac{\widehat{w}(u)}{\widehat{w}(z)}\right)^{1/p}
    \int_{\C^n}|\langle k_z,k_{\xi}\rangle_{\alpha}||\langle k_{\xi},k_u\rangle_{\alpha}|dv(\xi),
\end{align*}
which, together with \eqref{ele1}, implies that there exists $C>0$ such that
$$|\langle T_{\varphi}k^{(p,w)}_z,k^{(p',w')}_u\rangle_{\alpha}|
\lesssim\|\varphi\|_{L^{\infty}}C^{|u-z|}\int_{\C^n}|\langle k_z,k_{\xi}\rangle_{\alpha}||\langle k_{\xi},k_u\rangle_{\alpha}|dv(\xi).$$
Noting that
$$|\langle k_z,k_{\xi}\rangle_{\alpha}|=e^{-\frac{\alpha}{2}|\xi|^2}|k_z(\xi)|=e^{-\frac{\alpha}{2}|z-\xi|^2},$$
we have
\begin{equation}\label{zw}
|\langle T_{\varphi}k^{(p,w)}_z,k^{(p',w')}_u\rangle_{\alpha}|
\lesssim\|\varphi\|_{L^{\infty}}C^{|u-z|}\int_{\C^n}e^{-\frac{\alpha}{2}|z-\xi|^2-\frac{\alpha}{2}|\xi-u|^2}dv(\xi).
\end{equation}
Consequently, it follows from Fubini's theorem that
\begin{align*}
&\int_{\C^n}|\langle T_{\varphi}k^{(p,w)}_z,k^{(p',w')}_u\rangle_{\alpha}|dv(z)\\
&\ \lesssim\|\varphi\|_{L^{\infty}}\int_{\C^n}C^{|u-z|}\int_{\C^n}e^{-\frac{\alpha}{2}|z-\xi|^2-\frac{\alpha}{2}|\xi-u|^2}dv(\xi)dv(z)\\
&\ \leq\|\varphi\|_{L^{\infty}}\int_{\C^n}C^{|u-\xi|}e^{-\frac{\alpha}{2}|\xi-u|^2}
    \int_{\C^n}C^{|\xi-z|}e^{-\frac{\alpha}{2}|z-\xi|^2}dv(z)dv(\xi)\\
&\ \lesssim\|\varphi\|_{L^{\infty}}.
\end{align*}
Since $u\in\C^n$ is arbitrary, we conclude that $T_{\varphi}$ satisfies \eqref{sup2}. On the other hand, if $z\in B_r(u)^c$ for some $r>0$, then for any $\xi\in\C^n$, either $|z-\xi|\geq r/2$ or $|\xi-u|\geq r/2$, which implies that
$$e^{-\frac{\alpha}{2}|z-\xi|^2-\frac{\alpha}{2}|\xi-u|^2}\leq e^{-\frac{\alpha}{16}r^2-\frac{\alpha}{4}|z-\xi|^2-\frac{\alpha}{4}|\xi-u|^2}.$$
Hence similarly as before, \eqref{zw} gives that
\begin{align*}
&\int_{B_r(u)^c}|\langle T_{\varphi}k^{(p,w)}_z,k^{(p',w')}_u\rangle_{\alpha}|dv(z)\\
&\ \lesssim e^{-\frac{\alpha}{16}r^2}\|\varphi\|_{L^{\infty}}
    \int_{B_r(u)^c}C^{|u-z|}\int_{\C^n}e^{-\frac{\alpha}{4}|z-\xi|^2-\frac{\alpha}{4}|\xi-u|^2}dv(\xi)dv(z)\\
&\ \lesssim e^{-\frac{\alpha}{16}r^2}\|\varphi\|_{L^{\infty}},
\end{align*}
which, in conjunction with the arbitrariness of $u\in\C^n$, implies that $T_{\varphi}$ satisfies \eqref{vanishing2}. The proof is complete.
\end{proof}

\begin{proposition}\label{alg}
Let $\alpha>0$, $1<p<\infty$ and $w\in A^{\res}_p$. Then all weakly localized operators on $F^p_{\alpha,w}$ form an algebra.
\end{proposition}
\begin{proof}
It is sufficient to show that if $T$ and $S$ are weakly localized operators on $F^p_{\alpha,w}$, then $TS$ is weakly localized on $F^p_{\alpha,w}$. Suppose that $T$ and $S$ are weakly localized on $F^p_{\alpha,w}$. We only verify that $TS$ satisfies the condition \eqref{vanishing2}. The others are similar. Fix $\epsilon>0$. By the definition of weakly localized operators, there exists $r_0>0$ such that
\begin{equation}\label{ttt}
\sup_{u\in\C^n}\int_{B_{r_0/2}(u)^c}|\langle Tk^{(p,w)}_z,k^{(p',w')}_u\rangle_{\alpha}|dv(z)<\epsilon
\end{equation}
and
\begin{equation}\label{sss}
\sup_{u\in\C^n}\int_{B_{r_0/2}(u)^c}|\langle Sk^{(p,w)}_z,k^{(p',w')}_u\rangle_{\alpha}|dv(z)<\epsilon.
\end{equation}
For any $u\in\C^n$, applying Lemma \ref{equal} and the estimate \eqref{prod}, we have
\begin{align*}
I_{r_0}(u):&=\int_{B_{r_0}(u)^c}|\langle TSk^{(p,w)}_z,k^{(p',w')}_u\rangle_{\alpha}|dv(z)\\
&=\int_{B_{r_0}(u)^c}|\langle Sk^{(p,w)}_z,T^*k^{(p',w')}_u\rangle_{\alpha}|dv(z)\\
&=\int_{B_{r_0}(u)^c}\left|\int_{\C^n}Sk^{(p,w)}_z(\xi)\overline{T^*k^{(p',w')}_u(\xi)}d\lambda_{\alpha}(\xi)\right|dv(z)\\
&\leq\int_{B_{r_0}(u)^c}\int_{\C^n}|\langle Sk^{(p,w)}_z,K_{\xi}\rangle_{\alpha}|
    |\langle TK_{\xi},k^{(p',w')}_u\rangle_{\alpha}|d\lambda_{\alpha}(\xi)dv(z)\\
&\asymp\int_{B_{r_0}(u)^c}\int_{\C^n}|\langle Sk^{(p,w)}_z,k^{(p',w')}_{\xi}\rangle_{\alpha}|
    |\langle Tk^{(p,w)}_{\xi},k^{(p',w')}_u\rangle_{\alpha}|dv(\xi)dv(z)\\
&=\int_{B_{r_0/2}(u)}|\langle Tk^{(p,w)}_{\xi},k^{(p',w')}_u\rangle_{\alpha}|
    \int_{B_{r_0}(u)^c}|\langle Sk^{(p,w)}_z,k^{(p',w')}_{\xi}\rangle_{\alpha}|dv(z)dv(\xi)\\
&\quad+\int_{B_{r_0/2}(u)^c}|\langle Tk^{(p,w)}_{\xi},k^{(p',w')}_u\rangle_{\alpha}|
    \int_{B_{r_0}(u)^c}|\langle Sk^{(p,w)}_z,k^{(p',w')}_{\xi}\rangle_{\alpha}|dv(z)dv(\xi)\\
&=:I_{r_0,1}(u)+I_{r_0,2}(u).
\end{align*}
Note that for $\xi\in B_{r_0/2}(u)$, $B_{r_0}(u)^c\subset B_{r_0/2}(\xi)^c$. Consequently, by \eqref{sss} and the definition of weakly localized operators, we obtain that
\begin{align*}
I_{r_0,1}(u)&\leq\int_{B_{r_0/2}(u)}|\langle Tk^{(p,w)}_{\xi},k^{(p',w')}_u\rangle_{\alpha}|
    \int_{B_{r_0/2}(\xi)^c}|\langle Sk^{(p,w)}_z,k^{(p',w')}_{\xi}\rangle_{\alpha}|dv(z)dv(\xi)\\
&<\epsilon\int_{\C^n}|\langle Tk^{(p,w)}_{\xi},k^{(p',w')}_u\rangle_{\alpha}|dv(\xi)\\
&\lesssim\epsilon.
\end{align*}
Similarly, by \eqref{ttt}, we establish that
\begin{align*}
I_{r_0,2}(u)&\leq\int_{B_{r_0/2}(u)^c}|\langle Tk^{(p,w)}_{\xi},k^{(p',w')}_u\rangle_{\alpha}|
    \int_{\C^n}|\langle Sk^{(p,w)}_z,k^{(p',w')}_{\xi}\rangle_{\alpha}|dv(z)dv(\xi)\\
&\lesssim\int_{B_{r_0/2}(u)^c}|\langle Tk^{(p,w)}_{\xi},k^{(p',w')}_u\rangle_{\alpha}|dv(\xi)\\
&<\epsilon.
\end{align*}
Therefore,
$$I_{r_0}(u)=\int_{B_{r_0}(u)^c}|\langle TSk^{(p,w)}_z,k^{(p',w')}_u\rangle_{\alpha}|dv(z)\lesssim\epsilon.$$
Since $u\in\C^n$ is arbitrary, we conclude that $TS$ satisfies \eqref{vanishing2}.
\end{proof}

We will need the following Riesz--Kolmogorov type characterization of precompact sets in weighted Fock spaces, which is a direct application of \cite[Theorem 1.4]{MSWW}.

\begin{lemma}\label{RKtype}
Let $\alpha>0$, $1<p<\infty$ and $w\in A^{\res}_p$. A bounded set $\mathcal{S}\subset F^p_{\alpha,w}$ is precompact if and only if
$$\lim_{r\to\infty}\sup_{f\in\mathcal{S}}\int_{B_r(0)^c}|f(z)|^pe^{-\frac{p\alpha}{2}|z|^2}\widehat{w}(z)dv(z)=0.$$
\end{lemma}
\begin{proof}
For $u\in\C^n$, let
$$f_u(z)=\left(\frac{\alpha}{\pi}\right)^n\|K_u\|_{F^{p'}_{\alpha,w'}}e^{\alpha\langle z,u\rangle-\alpha|u|^2},\quad z\in\C^n.$$
We claim that $\big(\{f_u\}_{u\in\C^n},\{k^{(p',w')}_u\}_{u\in\C^n}\big)$ is a continuous frame for $F^p_{\alpha,w}$ with respect to $L^p(\C^n,dv)$ (see \cite{MSWW} for the definition). In fact, for any $f\in F^p_{\alpha,w}$ and $z\in\C^n$, Lemma \ref{equal} yields that
\begin{align*}
f(z)=P_{\alpha}f(z)&=\int_{\C^n}f(u)\overline{K_z(u)}d\lambda_{\alpha}(u)\\
&=\int_{\C^n}\langle f,K_u\rangle_{\alpha}K_u(z)d\lambda_{\alpha}(u)\\
&=\int_{\C^n}\langle f,k^{(p',w')}_u\rangle_{\alpha}\|K_u\|_{F^{p'}_{\alpha,w'}}K_u(z)d\lambda_{\alpha}(u)\\
&=\int_{\C^n}\langle f,k^{(p',w')}_u\rangle_{\alpha}f_u(z)dv(z).
\end{align*}
Moreover, by Lemma \ref{norm-eq} and \eqref{test'},
\begin{equation}\label{frame2}
\|f\|^p_{F^p_{\alpha,w}}\asymp\int_{\C^n}|f(u)|^pe^{-\frac{p\alpha}{2}|u|^2}\widehat{w}(u)dv(u)
\asymp\int_{\C^n}|\langle f,k^{(p',w')}_u\rangle_{\alpha}|^pdv(u),
\end{equation}
and $\{k^{(p',w')}_u\}_{u\in\C^n}$ is a bounded set in $F^{p'}_{\alpha,w'}=(F^p_{\alpha,w})^*$ due to Lemma \ref{dual}. Therefore, $\big(\{f_u\}_{u\in\C^n},\{k^{(p',w')}_u\}_{u\in\C^n}\big)$ is a continuous frame for $F^p_{\alpha,w}$ with respect to $L^p(\C^n,dv)$. Hence by \cite[Theorem 1.4]{MSWW}, $\mathcal{S}$ is precompact if and only if
$$\lim_{r\to\infty}\sup_{f\in\mathcal{S}}\int_{B_r(0)^c}|\langle f,k^{(p',w')}_u\rangle_{\alpha}|^pdv(u)=0,$$
which is equivalent to
$$\lim_{r\to\infty}\sup_{f\in\mathcal{S}}\int_{B_r(0)^c}|f(u)|^pe^{-\frac{p\alpha}{2}|u|^2}\widehat{w}(u)dv(u)=0.$$
The proof is complete.
\end{proof}

Based on Lemma \ref{RKtype}, we can obtain the following characterization of compact operators immediately.

\begin{corollary}\label{T-cpt}
Let $\alpha>0$, $1<p<\infty$ and $w\in A^{\res}_p$. Suppose that $T$ is a bounded linear operator on $F^p_{\alpha,w}$. Then $T$ is compact if and only if
$$\lim_{r\to\infty}\sup_{\|f\|_{F^p_{\alpha,w}}\leq1}\int_{B_r(0)^c}|Tf(z)|^pe^{-\frac{p\alpha}{2}|z|^2}\widehat{w}(z)dv(z)=0.$$
\end{corollary}

The following lemma is the last piece in the proof of Theorem \ref{main1}.

\begin{lemma}\label{cpt-suff}
Let $\alpha>0$, $1<p<\infty$ and $w\in A^{\res}_p$. Suppose that $T$ is a bounded linear operator on $F^p_{\alpha,w}$ such that \eqref{sup2} holds. If
\begin{equation}\label{limit0}
\lim_{r\to\infty}\sup_{z\in\C^n}\int_{B_r(0)^c}|\langle Tk^{(p,w)}_z,k^{(p',w')}_u\rangle_{\alpha}|dv(u)=0,
\end{equation}
then $T$ is compact on $F^p_{\alpha,w}$.
\end{lemma}
\begin{proof}
Fix $f\in F^p_{\alpha,w}$ with $\|f\|_{F^p_{\alpha,w}}\leq1$. Then for any $z\in\C^n$, Lemma \ref{equal} gives that
\begin{align*}
|Tf(z)|&=|\langle Tf,K_z\rangle_{\alpha}|\\
&=\left|\int_{\C^n}f(u)\overline{T^*K_z(u)}d\lambda_{\alpha}(u)\right|\\
&=\left|\int_{\C^n}\langle f,K_u\rangle_{\alpha}\langle TK_u,K_z\rangle_{\alpha}d\lambda_{\alpha}(u)\right|,
\end{align*}
which, combined with \eqref{prod}, implies that
\begin{align*}
|Tf(z)|&\lesssim\int_{\C^n}|\langle f,k^{(p',w')}_u\rangle_{\alpha}||\langle Tk^{(p,w)}_u,K_z\rangle_{\alpha}|dv(u)\\
&=\|K_z\|_{F^{p'}_{\alpha,w'}}\int_{\C^n}|\langle f,k^{(p',w')}_u\rangle_{\alpha}||\langle Tk^{(p,w)}_u,k^{(p',w')}_z\rangle_{\alpha}|dv(u).
\end{align*}
By H\"{o}lder's inequality and the condition \eqref{sup2},
\begin{align}\label{Tf}
|Tf(z)|^p
&\lesssim\|K_z\|^p_{F^{p'}_{\alpha,w'}}\left(\int_{\C^n}|\langle f,k^{(p',w')}_u\rangle_{\alpha}|
    |\langle Tk^{(p,w)}_u,k^{(p',w')}_z\rangle_{\alpha}|dv(u)\right)^p\nonumber\\
&\leq\|K_z\|^p_{F^{p'}_{\alpha,w'}}\left(\int_{\C^n}|\langle Tk^{(p,w)}_u,k^{(p',w')}_z\rangle_{\alpha}|dv(u)\right)^{p/p'}\cdot\nonumber\\
&    \qquad\qquad\left(\int_{\C^n}|\langle f,k^{(p',w')}_u\rangle_{\alpha}|^p|\langle Tk^{(p,w)}_u,k^{(p',w')}_z\rangle_{\alpha}|dv(u)\right)
    \nonumber\\
&\lesssim\|K_z\|^p_{F^{p'}_{\alpha,w'}}
    \int_{\C^n}|\langle f,k^{(p',w')}_u\rangle_{\alpha}|^p|\langle Tk^{(p,w)}_u,k^{(p',w')}_z\rangle_{\alpha}|dv(u)\nonumber\\
&\asymp e^{\frac{p\alpha}{2}|z|^2}\widehat{w}(z)^{-1}
    \int_{\C^n}|\langle f,k^{(p',w')}_u\rangle_{\alpha}|^p|\langle Tk^{(p,w)}_u,k^{(p',w')}_z\rangle_{\alpha}|dv(u),
\end{align}
where the last step is due to \eqref{test'}.
For any $\epsilon>0$, we can use \eqref{limit0} to find $r_0>0$ such that
$$\sup_{u\in\C^n}\int_{B_r(0)^c}|\langle Tk^{(p,w)}_u,k^{(p',w')}_z\rangle_{\alpha}|dv(z)<\epsilon$$
whenever $r>r_0$. Consequently, the estimate \eqref{Tf} together with Fubini's theorem and \eqref{frame2} yields that for $r>r_0$,
\begin{align*}
&\int_{B_r(0)}|Tf(z)|^pe^{-\frac{p\alpha}{2}|z|^2}\widehat{w}(z)dv(z)\\
&\ \lesssim\int_{B_r(0)}\int_{\C^n}|\langle f,k^{(p',w')}_u\rangle_{\alpha}|^p|\langle Tk^{(p,w)}_u,k^{(p',w')}_z\rangle_{\alpha}|dv(u)dv(z)\\
&\ =\int_{\C^n}|\langle f,k^{(p',w')}_u\rangle_{\alpha}|^p\int_{B_r(0)^c}|\langle Tk^{(p,w)}_u,k^{(p',w')}_z\rangle_{\alpha}|dv(z)dv(u)\\
&\ <\epsilon\int_{\C^n}|\langle f,k^{(p',w')}_u\rangle_{\alpha}|^pdv(u)\\
&\ \lesssim\epsilon.
\end{align*}
Since $f$ is arbitrary in the closed unit ball of $F^p_{\alpha,w}$, we deduce from Corollary \ref{T-cpt} that $T$ is compact on $F^p_{\alpha,w}$.
\end{proof}

We are now ready to prove Theorem \ref{main1}.

\begin{proof}[Proof of Theorem \ref{main1}]
It is clear that $T$ is bounded on $F^p_{\alpha,w}$. By Propositions \ref{wl} and \ref{alg}, $T$ is weakly localized. Suppose first that $\lim_{|z|\to\infty}\widetilde{T}(z)=0$. Since $T$ is bounded on the Fock space $F^p_{\alpha}$, we may use \cite[Proposition 1.5]{Is15} to deduce that for each $r>0$,
$$\lim_{|z|\to\infty}\sup_{u\in B_r(z)}|\langle Tk_z,k_u\rangle_{\alpha}|=0.$$
Fix $\epsilon>0$. Since $T$ is weakly localized, there exists $r_0>0$ such that
\begin{equation}\label{little}
\sup_{z\in\C^n}\int_{B_{r_0}(z)^c}|\langle Tk^{(p,w)}_z,k^{(p',w')}_u\rangle_{\alpha}|dv(u)<\frac{\epsilon}{2}.
\end{equation}
For the above $r_0$, there exists $r_1>0$ such that for $|z|>r_1$,
\begin{equation}\label{inner}
\sup_{u\in B_{r_0}(z)}|\langle Tk_z,k_u\rangle_{\alpha}|<\frac{\epsilon}{2v(B_{r_0}(0))C^{r_0/p}},
\end{equation}
where $C$ is the constant provided by \eqref{ele1} such that for any $\zeta,\xi\in\C^n$,
$$\frac{\widehat{w}(\zeta)}{\widehat{w}(\xi)}\lesssim C^{|\zeta-\xi|}.$$
Suppose now that $r>r_0+r_1$. For any $z\in\C^n$, \eqref{little} gives that
\begin{align*}
&\int_{B_r(0)^c}|\langle Tk^{(p,w)}_z,k^{(p',w')}_u\rangle_{\alpha}|dv(u)\\
&\ \leq\int_{B_{r_0+r_1}(0)^c}|\langle Tk^{(p,w)}_z,k^{(p',w')}_u\rangle_{\alpha}|dv(u)\\
&\ =\int_{\big(B_{r_0+r_1}(0)\cup B_{r_0}(z)\big)^c}|\langle Tk^{(p,w)}_z,k^{(p',w')}_u\rangle_{\alpha}|dv(u)\\
&\qquad +\int_{B_{r_0+r_1}(0)^c\cap B_{r_0}(z)}|\langle Tk^{(p,w)}_z,k^{(p',w')}_u\rangle_{\alpha}|dv(u)\\
&\ <\frac{\epsilon}{2}+\int_{B_{r_0+r_1}(0)^c\cap B_{r_0}(z)}|\langle Tk^{(p,w)}_z,k^{(p',w')}_u\rangle_{\alpha}|dv(u).
\end{align*}
If $|z|\leq r_1$, then $B_{r_0+r_1}(0)^c\cap B_{r_0}(z)=\emptyset$ and the second term above is trivial. If $|z|>r_1$, then by Lemma \ref{test}, \eqref{test'} and \eqref{inner},
\begin{align*}
&\int_{B_{r_0+r_1}(0)^c\cap B_{r_0}(z)}|\langle Tk^{(p,w)}_z,k^{(p',w')}_u\rangle_{\alpha}|dv(u)\\
&\ \lesssim\int_{B_{r_0}(z)}|\langle Tk_z,k_u\rangle_{\alpha}|\left(\frac{\widehat{w}(u)}{\widehat{w}(z)}\right)^{1/p}dv(u)\\
&\ <\frac{\epsilon}{2v(B_{r_0}(0))C^{r_0/p}}\int_{B_{r_0}(z)}C^{\frac{1}{p}|u-z|}dv(u)\\
&\ <\frac{\epsilon}{2}.
\end{align*}
Therefore,
$$\sup_{z\in\C^n}\int_{B_r(0)^c}|\langle Tk^{(p,w)}_z,k^{(p',w')}_u\rangle_{\alpha}|dv(u)\lesssim\epsilon.$$
That is,
$$\lim_{r\to\infty}\sup_{z\in\C^n}\int_{B_r(0)^c}|\langle Tk^{(p,w)}_z,k^{(p',w')}_u\rangle_{\alpha}|dv(u)=0,$$
which, combined with Lemma \ref{cpt-suff}, implies that $T$ is compact on $F^p_{\alpha,w}$.

Conversely, suppose that $T$ is compact on $F^p_{\alpha,w}$. Since $\{k^{(p,w)}_z\}$ is bounded in $F^p_{\alpha,w}$ and converges to $0$ uniformly on compact subsets of $\C^n$ as $|z|\to\infty$ (see \cite[Lemma 2.3]{Ch24}), it is easy to see that $k^{(p,w)}_z\to0$ weakly in $F^p_{\alpha,w}$ as $|z|\to\infty$. Therefore, by \eqref{prod}, H\"{o}lder's inequality and the compactness of $T$,
$$\left|\widetilde{T}(z)\right|=|\langle Tk_z,k_z\rangle_{\alpha}|
\asymp|\langle Tk^{(p,w)}_z,k^{(p',w')}_z\rangle_{\alpha}|\leq\|Tk^{(p,w)}_z\|_{F^p_{\alpha,w}}\to0$$
as $|z|\to\infty$. The proof is complete.
\end{proof}

\section{Weighted boundedness}\label{wbdd}

In this section, we are going to prove Theorems \ref{main2} and \ref{main3}.

To establish the necessary parts of Theorem \ref{main2}, we consider a class of rank one operators induced by the normalized reproducing kernels of $F^2_{\alpha}$. Given $\varphi\in L^{\infty}$, $u\in\C^n$ and $r>0$, the operator $T^{(u,r)}_{\varphi}$ is defined for $f\in L^1_{\mathrm{loc}}(\C^n)$ by
$$T^{(u,r)}_{\varphi}f:=\chi_{Q_r(u)}k_u\int_{Q_r(u)}\varphi f\overline{k_u}d\lambda_{\alpha},$$
where $\chi_{E}$ denotes the characteristic function of the set $E$. The following lemma can be proved by the same method as in \cite[Proposition 3.2]{CW24-1}. We here omit the details.

\begin{lemma}\label{Tur}
Let $\alpha,r>0$, $1\leq p<\infty$, $\varphi\in L^{\infty}$, and let $w$ be a weight on $\C^n$. Suppose that $T_{\varphi}$ is bounded on $L^p_{\alpha,w}$. Then for any $u\in\C^n$, $T^{(u,r)}_{\varphi}$ is bounded on $L^p_{\alpha,w}$, and
$$\|T^{(u,r)}_{\varphi}\|\leq e^{\frac{n\alpha r^2}{2}}\|T_{\varphi}\|.$$
\end{lemma}

The following proposition gives the necessary conditions for $T_{\varphi}$ to be bounded on the weighted spaces $L^p_{\alpha,w}$.

\begin{proposition}\label{nece}
Let $\alpha>0$, $1\leq p<\infty$, $\varphi\in L^{\infty}$, and let $w$ be a weight on $\C^n$. Suppose that $T_{\varphi}$ is bounded on $L^p_{\alpha,w}$.
\begin{enumerate}
	\item [(1)] If $1<p<\infty$, then for any $r>0$,
	$$\sup_{Q:l(Q)=r}\left(\frac{1}{v(Q)}\int_{Q}wdv\right)
	\left(\frac{1}{v(Q)}\int_{Q}|\varphi|^{p'}w^{-\frac{p'}{p}}dv\right)^{\frac{p}{p'}}<\infty.$$
	\item [(2)] If $p=1$, then for any $r>0$,
	$$\sup_{Q:l(Q)=r}\left(\frac{1}{v(Q)}\int_{Q}wdv\right)\|\varphi w^{-1}\|_{L^{\infty}(Q)}<\infty.$$
\end{enumerate}
\end{proposition}
\begin{proof}
(1) Let $r>0$. We only need to consider the points $u\in\C^n$ such that $\int_{Q_r(u)}|\varphi|^{p'}w^{-\frac{p'}{p}}dv>0$.
Fix such a point $u\in\C^n$. For any sufficiently large $m>0$, let $E_m$ be the set
$$E_m:=\left\{\xi\in\C^n:|\varphi(\xi)|^{p'}w'(\xi)\leq m\right\}.$$
Write $\delta=\frac{p-2}{p-1}$ and define
\begin{equation*}
f_{m,u,r}(\xi):=
\begin{cases}
	\frac{\overline{\varphi(\xi)}}{|\varphi(\xi)|^{\delta}}w'(\xi)k_u(\xi)\chi_{Q_r(u)\cap E_m}(\xi),&   \text{if $\varphi(\xi)\neq0$},\\
	0,&   \text{if $\varphi(\xi)=0$}.
\end{cases}
\end{equation*}
Then
\begin{align*}
\|f_{m,u,r}\|_{L^p_{\alpha,w}}&=\left(\int_{Q_r(u)\cap E_m}|\varphi(\xi)|^{p(1-\delta)}w'(\xi)^p|k_u(\xi)|^p
    e^{-\frac{p\alpha}{2}|\xi|^2}w(\xi)dv(\xi)\right)^{1/p}\\
&=\left(\int_{Q_r(u)\cap E_m}|\varphi(\xi)|^{p'}e^{-\frac{p\alpha}{2}|u-\xi|^2}w'(\xi)dv(\xi)\right)^{1/p}\\
&\leq\left(\int_{Q_r(u)\cap E_m}|\varphi|^{p'}w'dv\right)^{1/p}.
\end{align*}
On the other hand, for any $z\in\C^n$,
\begin{align*}
T^{(u,r)}_{\varphi}f_{m,u,r}(z)
&=\chi_{Q_r(u)}(z)k_u(z)\int_{Q_r(u)}\varphi(\xi)f_{m,u,r}(\xi)\overline{k_{u}(\xi)}d\lambda_{\alpha}(\xi)\\
&=\left(\frac{\alpha}{\pi}\right)^n\chi_{Q_r(u)}(z)k_u(z)
    \int_{Q_r(u)}\varphi(\xi)f_{m,u,r}(\xi)e^{\alpha\langle u,\xi\rangle-\frac{\alpha}{2}|u|^2-\alpha|\xi|^2}dv(\xi)\\
&=\left(\frac{\alpha}{\pi}\right)^n\chi_{Q_r(u)}(z)k_u(z)
    \int_{Q_r(u)\cap E_m}|\varphi(\xi)|^{2-\delta}e^{-\alpha|u-\xi|^2}w'(\xi)dv(\xi),
\end{align*}
which implies that
$$|T^{(u,r)}_{\varphi}f_{m,u,r}(z)|\geq\left(\frac{\alpha}{\pi}\right)^ne^{-\frac{n\alpha r^2}{2}}\chi_{Q_r(u)}(z)|k_u(z)|
\int_{Q_r(u)\cap E_m}|\varphi|^{p'}w'dv.$$
Consequently,
\begin{align*}
\|T^{(u,r)}_{\varphi}f_{m,u,r}\|_{L^p_{\alpha,w}}
&\geq\left(\frac{\alpha}{\pi}\right)^ne^{-\frac{n\alpha r^2}{2}}\int_{Q_r(u)\cap E_m}|\varphi|^{p'}w'dv\\
&\qquad\qquad    \times\left(\int_{Q_r(u)}|k_u(z)|^pe^{-\frac{p\alpha}{2}|z|^2}w(z)dv(z)\right)^{1/p}\\
&\geq\left(\frac{\alpha}{\pi}\right)^ne^{-\frac{3n\alpha r^2}{4}}\left(\int_{Q_r(u)\cap E_m}|\varphi|^{p'}w'dv\right)
    \left(\int_{Q_r(u)}wdv\right)^{1/p}.
\end{align*}
Combining the above estimates with Lemma \ref{Tur} yields that
\begin{align*}
\left(\frac{\alpha}{\pi}\right)^ne^{-\frac{3n\alpha r^2}{4}}&\left(\int_{Q_r(u)\cap E_m}|\varphi|^{p'}w'dv\right)
    \left(\int_{Q_r(u)}wdv\right)^{1/p}\\
&\qquad\leq e^{\frac{n\alpha r^2}{2}}\|T_{\varphi}\|\left(\int_{Q_r(u)\cap E_m}|\varphi|^{p'}w'dv\right)^{1/p},
\end{align*}
which is equivalent to
$$\left(\int_{Q_r(u)}wdv\right)^{1/p}\left(\int_{Q_r(u)\cap E_m}|\varphi|^{p'}w'dv\right)^{1/p'}
\leq \left(\frac{\pi}{\alpha}\right)^ne^{\frac{5n\alpha r^2}{4}}\|T_{\varphi}\|.$$
Letting $m\to\infty$ and using Fatou's lemma, we obtain that
$$\left(\int_{Q_r(u)}wdv\right)^{1/p}\left(\int_{Q_r(u)}|\varphi|^{p'}w'dv\right)^{1/p'}
\leq \left(\frac{\pi}{\alpha}\right)^ne^{\frac{5n\alpha r^2}{4}}\|T_{\varphi}\|.$$
Since $u\in\C^n$ is arbitrary, we finally conclude that
$$\sup_{u\in\C^n}\left(\int_{Q_r(u)}wdv\right)^{1/p}\left(\int_{Q_r(u)}|\varphi|^{p'}w'dv\right)^{1/p'}
\leq \left(\frac{\pi}{\alpha}\right)^ne^{\frac{5n\alpha r^2}{4}}\|T_{\varphi}\|,$$
and the desired result is established.

(2) Fix $r>0$. Note that under the pairing
$$\langle f,g\rangle_{L^2_{\alpha/2,w}}:=\int_{\C^n}f(z)\overline{g(z)}e^{-\frac{\alpha}{2}|z|^2}w(z)dv(z),$$
the dual space of $L^1_{\alpha,w}$ can be identified with $L^{\infty}$, and the adjoint of $T_{\varphi}$ is given by
$$T^*_{\varphi}f(z)=\left(\frac{\alpha}{\pi}\right)^n\frac{\overline{\varphi(z)}}{w(z)}\int_{\C^n}f(\xi)
e^{\alpha\langle z,\xi\rangle-\frac{\alpha}{2}|z|^2-\frac{\alpha}{2}|\xi|^2}w(\xi)dv(\xi).$$
It follows from the boundedness of $T_{\varphi}$ on $L^1_{\alpha,w}$ that for any $f\in L^{\infty}$ and almost every $z\in\C^n$,
$$\left(\frac{\alpha}{\pi}\right)^n\frac{|\varphi(z)|}{w(z)}
  \left|\int_{\C^n}f(\xi)e^{\alpha\langle z,\xi\rangle-\frac{\alpha}{2}|z|^2-\frac{\alpha}{2}|\xi|^2}w(\xi)dv(\xi)\right|
  \leq\|T_{\varphi}\|\|f\|_{L^{\infty}}.$$
Fix $u\in\C^n$, and for $z\in Q_r(u)$, define $f(\xi)=e^{\frac{\alpha}{2}\langle\xi,z\rangle-\frac{\alpha}{2}\langle z,\xi\rangle}$. Then we obtain that for almost every $z\in Q_r(u)$,
\begin{align*}
\|T_{\varphi}\|&\geq\left(\frac{\alpha}{\pi}\right)^n\frac{|\varphi(z)|}{w(z)}\int_{\C^n}e^{-\frac{\alpha}{2}|z-\xi|^2}w(\xi)dv(\xi)\\
&\geq\left(\frac{\alpha}{\pi}\right)^n\frac{|\varphi(z)|}{w(z)}\int_{Q_r(u)}e^{-\frac{\alpha}{2}|z-\xi|^2}w(\xi)dv(\xi)\\
&\geq\left(\frac{\alpha}{\pi}\right)^ne^{-n\alpha r^2}\frac{|\varphi(z)|}{w(z)}\int_{Q_r(u)}w(\xi)dv(\xi),
\end{align*}
which is equivalent to
$$\left(\frac{1}{v(Q_r(u))}\int_{Q_r(u)}wdv\right)\|\varphi w^{-1}\|_{L^{\infty}(Q_r(u))}
\leq\left(\frac{\pi}{\alpha r^2}\right)^ne^{n\alpha r^2}\|T_{\varphi}\|.$$
The arbitrariness of $u\in\C^n$ finishes the proof.
\end{proof}

We are now in a position to prove Theorem \ref{main2}.

\begin{proof}[Proof of Theorem \ref{main2}]
The necessary parts follow from Proposition \ref{nece}. We now consider the sufficient parts. Write $\tilde{f}(u)=f(u)e^{-\frac{\alpha}{2}|u|^2}$ for simplicity.

(1) Note that if $\nu,\nu'\in r\Z^{2n}$ and $z\in Q_r(\nu)$, $u\in Q_r(\nu')$, then
$$e^{-\frac{\alpha}{2}|z-u|^2}\leq e^{-\frac{\alpha}{4}|\nu-\nu'|^2+n\alpha r^2}.$$
Hence for any $f\in L^p_{\alpha,w}$,
\begin{align*}
\|T_{\varphi}f\|^p_{L^p_{\alpha,w}}
&=\int_{\C^n}\left|\int_{\C^n}\varphi(u)f(u)e^{\alpha\langle z,u\rangle}d\lambda_{\alpha}(u)\right|^pe^{-\frac{p\alpha}{2}|z|^2}w(z)dv(z)\\
&\lesssim\int_{\C^n}\left(\int_{\C^n}|\varphi(u)||f(u)|e^{-\frac{\alpha}{2}|u|^2}e^{-\frac{\alpha}{2}|z-u|^2}dv(u)\right)^pw(z)dv(z)\\
&=\sum_{\nu\in r\Z^{2n}}\int_{Q_r(\nu)}\left(\sum_{\nu'\in r\Z^{2n}}
    \int_{Q_r(\nu')}|\varphi(u)||\tilde{f}(u)|e^{-\frac{\alpha}{2}|z-u|^2}dv(u)\right)^pw(z)dv(z)\\
&\lesssim\sum_{\nu\in r\Z^{2n}}w(Q_r(\nu))\left(\sum_{\nu'\in r\Z^{2n}}e^{-\frac{\alpha}{4}|\nu-\nu'|^2}
    \int_{Q_r(\nu')}|\varphi(u)||\tilde{f}(u)|dv(u)\right)^p,
\end{align*}
which, in conjunction with H\"{o}lder's inequality and Fubini's theorem, implies that
\begin{align*}
\|T_{\varphi}f\|^p_{L^p_{\alpha,w}}
&\lesssim\sum_{\nu\in r\Z^{2n}}w(Q_r(\nu))\sum_{\nu'\in r\Z^{2n}}e^{-\frac{p\alpha}{8}|\nu-\nu'|^2}
    \left(\int_{Q_r(\nu')}|\varphi||\tilde{f}|dv(u)\right)^p\\
&\leq\sum_{\nu\in r\Z^{2n}}w(Q_r(\nu))\sum_{\nu'\in r\Z^{2n}}e^{-\frac{p\alpha}{8}|\nu-\nu'|^2}\\
&\qquad\qquad    \times\left(\int_{Q_r(\nu')}|\varphi|^{p'}w^{-\frac{p'}{p}}dv\right)^{\frac{p}{p'}}
    \left(\int_{Q_r(\nu')}|\tilde{f}|^pwdv\right)\\
&=\sum_{\nu'\in r\Z^{2n}}\sum_{\nu\in r\Z^{2n}}w(Q_r(\nu))e^{-\frac{p\alpha}{8}|\nu-\nu'|^2}\\
&\qquad\qquad\times\left(\int_{Q_r(\nu')}|\varphi|^{p'}w^{-\frac{p'}{p}}dv\right)^{\frac{p}{p'}}
    \left(\int_{Q_r(\nu')}|\tilde{f}|^pwdv\right).
\end{align*}
Since $w$ is $r$-doubling, by Lemma \ref{ele}, there exists $C>0$ such that for any $\nu,\nu'\in r\Z^{2n}$, $w(Q_r(\nu))\leq C^{|\nu-\nu'|}w(Q_r(\nu'))$. Consequently, for each $\nu'\in r\Z^{2n}$,
\begin{align*}
&\sum_{\nu\in r\Z^{2n}}w(Q_r(\nu))e^{-\frac{p\alpha}{8}|\nu-\nu'|^2}
  \left(\int_{Q_r(\nu')}|\varphi|^{p'}w^{-\frac{p'}{p}}dv\right)^{\frac{p}{p'}}\\
&\ \ \leq\sum_{\nu\in r\Z^{2n}}C^{|\nu-\nu'|}e^{-\frac{p\alpha}{8}|\nu-\nu'|^2}
  w(Q_r(\nu'))\left(\int_{Q_r(\nu')}|\varphi|^{p'}w^{-\frac{p'}{p}}dv\right)^{\frac{p}{p'}}\\
&\ \ \lesssim\sup_{Q:l(Q)=r}\left(\frac{1}{v(Q)}\int_{Q}wdv\right)
  \left(\frac{1}{v(Q)}\int_{Q}|\varphi|^{p'}w^{-\frac{p'}{p}}dv\right)^{\frac{p}{p'}}.
\end{align*}
Therefore,
$$\|T_{\varphi}f\|^p_{L^p_{\alpha,w}}\lesssim
\sup_{Q:l(Q)=r}\left(\frac{1}{v(Q)}\int_{Q}wdv\right)
\left(\frac{1}{v(Q)}\int_{Q}|\varphi|^{p'}w^{-\frac{p'}{p}}dv\right)^{\frac{p}{p'}}\|f\|^p_{L^p_{\alpha,w}},$$
and the desired boundedness follows.

(2) For any $f\in L^1_{\alpha,w}$, we apply Lemma \ref{ele} again to obtain that
\begin{align*}
\|T_{\varphi}f\|_{L^1_{\alpha,w}}
&=\int_{\C^n}\left|\int_{\C^n}\varphi(u)f(u)e^{\alpha\langle z,u\rangle}d\lambda_{\alpha}(\xi)\right|
  e^{-\frac{\alpha}{2}|z|^2}w(z)dv(z)\\
&\leq\left(\frac{\alpha}{\pi}\right)^n\int_{\C^n}\int_{\C^n}|\varphi(u)||f(u)|e^{-\frac{\alpha}{2}|u|^2}e^{-\frac{\alpha}{2}|z-u|^2}
  dv(u)w(z)dv(z)\\
&\lesssim\sum_{\nu\in r\Z^{2n}}w(Q_r(\nu))\sum_{\nu'\in r\Z^{2n}}e^{-\frac{\alpha}{4}|\nu-\nu'|^2}
  \int_{Q_r(\nu')}|\varphi(u)||\tilde{f}(u)|dv(u)\\
&\leq\sum_{\nu'\in r\Z^{2n}}\sum_{\nu\in r\Z^{2n}}e^{-\frac{\alpha}{4}|\nu-\nu'|^2}C^{|\nu-\nu'|}w(Q_r(\nu'))
  \int_{Q_r(\nu')}|\varphi||\tilde{f}|dv\\
&\lesssim\sum_{\nu'\in r\Z^{2n}}w(Q_r(\nu'))\|\varphi w^{-1}\|_{L^{\infty}(Q_r(\nu'))}
  \int_{Q_r(\nu')}|\tilde{f}|wdv\\
&\lesssim\sup_{Q:l(Q)=r}\left(\frac{1}{v(Q)}\int_{Q}wdv\right)\|\varphi w^{-1}\|_{L^{\infty}(Q)}\|f\|_{L^1_{\alpha,w}}.
\end{align*}
The proof is complete.
\end{proof}

The proof of Theorem \ref{main3} is similar to that of Theorem \ref{main2}, so we only give a sketch of it.

\begin{proof}[Proof of Theorem \ref{main3}]
Suppose first that $P_{\alpha}:L^p_{\alpha,\sigma}\to L^p_{\alpha,w}$ is bounded. We first consider the case $1<p<\infty$. Fix $u\in\C^n$, and define
$$P_{\alpha,u,r}f:=\chi_{Q_r(u)}k_u\int_{Q_r(u)}f\overline{k_u}d\lambda_{\alpha}.$$
Then by the same method as in \cite[Proposition 3.2]{CW24-1}, $P_{\alpha,u,r}:L^p_{\alpha,\sigma}\to L^p_{\alpha,w}$ is bounded, and
$$\|P_{\alpha,u,r}\|_{L^p_{\alpha,\sigma}\to L^p_{\alpha,w}}\leq e^{\frac{n\alpha r^2}{2}}\|P_{\alpha}\|_{L^p_{\alpha,\sigma}\to L^p_{\alpha,w}}.$$
For any sufficiently large $m>0$, let
$$E_m=\left\{\xi\in\C^n:\sigma(\xi)^{-\frac{p'}{p}}\leq m\right\},$$
and define
$$f_{m,u,r}(\xi)=k_u(\xi)\sigma(\xi)^{-\frac{p'}{p}}\chi_{Q_r(u)\cap E_m}(\xi).$$
Then
$$\|f_{m,u,r}\|_{L^p_{\alpha,\sigma}}\leq\left(\int_{Q_r(u)\cap E_m}\sigma^{-\frac{p'}{p}}dv\right)^{1/p},$$
and
$$\|P_{\alpha,u,r}f_{m,u,r}\|_{L^p_{\alpha,w}}\geq
\left(\frac{\alpha}{\pi}\right)^ne^{-\frac{3n\alpha r^2}{4}}\left(\int_{Q_r(u)\cap E_m}\sigma^{-\frac{p'}{p}}dv\right)
  \left(\int_{Q_r(u)}wdv\right)^{1/p}.$$
Therefore,
\begin{align*}
\left(\frac{\alpha}{\pi}\right)^ne^{-\frac{3n\alpha r^2}{4}}&\left(\int_{Q_r(u)\cap E_m}\sigma^{-\frac{p'}{p}}dv\right)
    \left(\int_{Q_r(u)}wdv\right)^{1/p}\\
&\qquad\leq e^{\frac{n\alpha r^2}{2}}\|P_{\alpha}\|\left(\int_{Q_r(u)\cap E_m}\sigma^{-\frac{p'}{p}}dv\right)^{1/p},
\end{align*}
which is equivalent to
$$\left(\int_{Q_r(u)}wdv\right)^{1/p}\left(\int_{Q_r(u)\cap E_m}\sigma^{-\frac{p'}{p}}dv\right)^{1/p'}
\leq\left(\frac{\pi}{\alpha}\right)^ne^{\frac{5n\alpha r^2}{4}}\|P_{\alpha}\|.$$
Letting $m\to\infty$, and using Fatou's lemma and the arbitrariness of $u\in\C^n$, we obtain that
$$[w,\sigma]_{A_{p,r}}^{1/p}\leq \left(\frac{\pi}{\alpha r^2}\right)^ne^{\frac{5n\alpha r^2}{4}}\|P_{\alpha}\|.$$
We now consider the case $p=1$. Note that $(L^1_{\alpha,\sigma})^*=L^{\infty}$ under the pairing $\langle\cdot,\cdot\rangle_{L^2_{\alpha/2,\sigma}}$, $(L^1_{\alpha,w})^*=L^{\infty}$ under the pairing $\langle\cdot,\cdot\rangle_{L^2_{\alpha/2,w}}$, and the adjoint of $P_{\alpha}$ with respect to the above pairings is given by
$$P^*_{\alpha}f(z)=\left(\frac{\alpha}{\pi}\right)^n\frac{1}{\sigma(z)}\int_{\C^n}f(\xi)
  e^{\alpha\langle z,\xi\rangle-\frac{\alpha}{2}|\xi|^2-\frac{\alpha}{2}|z|^2}w(\xi)dv(\xi).$$
Hence by the boundedness of $P_{\alpha}:L^1_{\alpha,\sigma}\to L^1_{\alpha,w}$, we obtain that for any $f\in L^{\infty}$ and almost every $z\in\C^n$,
$$\left(\frac{\alpha}{\pi}\right)^n\frac{1}{\sigma(z)}\left|\int_{\C^n}f(\xi)
e^{\alpha\langle z,\xi\rangle-\frac{\alpha}{2}|\xi|^2-\frac{\alpha}{2}|z|^2}w(\xi)dv(\xi)\right|
\leq\|P_{\alpha}\|\|f\|_{L^{\infty}}.$$
Arguing as in the proof of Proposition \ref{nece}, we establish that
$$[w,\sigma]_{A_{1,r}}\leq\left(\frac{\pi}{\alpha r^2}\right)^ne^{n\alpha r^2}\|P_{\alpha}\|.$$

Conversely, suppose that $[w,\sigma]_{A_{p,r}}<\infty$. Fix $f\in L^p_{\alpha,\sigma}$ and write $$\tilde{f}(u)=f(u)e^{-\frac{\alpha}{2}|u|^2}.$$
In the case $1<p<\infty$, we use H\"{o}lder's inequality and Lemma \ref{ele} to obtain that
\begin{align*}
\|P_{\alpha}f\|^p_{L^p_{\alpha,w}}
&=\int_{\C^n}\left|\int_{\C^n}f(u)e^{\alpha\langle z,u\rangle}d\lambda_{\alpha}(u)\right|^pe^{-\frac{p\alpha}{2}|z|^2}w(z)dv(z)\\
&\lesssim\sum_{\nu\in r\Z^{2n}}w(Q_r(\nu))\left(\sum_{\nu'\in r\Z^{2n}}e^{-\frac{\alpha}{4}|\nu-\nu'|^2}
	\int_{Q_r(\nu')}|\tilde{f}|dv\right)^p\\
&\lesssim\sum_{\nu\in r\Z^{2n}}w(Q_r(\nu))\sum_{\nu'\in r\Z^{2n}}e^{-\frac{p\alpha}{8}|\nu-\nu'|^2}\left(\int_{Q_r(\nu')}|\tilde{f}|dv\right)^p\\
&\leq\sum_{\nu'\in r\Z^{2n}}\sum_{\nu\in r\Z^{2n}}e^{-\frac{p\alpha}{8}|\nu-\nu'|^2}w(Q_r(\nu))\\
&\qquad\qquad\times\left(\int_{Q_r(\nu')}\sigma^{-\frac{p'}{p}}dv\right)^{\frac{p}{p'}}\int_{Q_r(\nu')}|\tilde{f}|^p\sigma dv\\
&\leq\sum_{\nu'\in r\Z^{2n}}\sum_{\nu\in r\Z^{2n}}e^{-\frac{p\alpha}{8}|\nu-\nu'|^2}C^{|\nu-\nu'|}\\
&\qquad\qquad\times\left(\int_{Q_r(\nu')}wdv\right)\left(\int_{Q_r(\nu')}\sigma^{-\frac{p'}{p}}dv\right)^{\frac{p}{p'}}
    \int_{Q_r(\nu')}|\tilde{f}|^p\sigma dv\\
&\lesssim[w,\sigma]_{A_{p,r}}\|f\|^p_{L^p_{\alpha,\sigma}}.
\end{align*}

In the case $p=1$, we apply Lemma \ref{ele} to establish that
\begin{align*}
\|P_{\alpha}f\|_{L^1_{\alpha,w}}&=\int_{\C^n}\left|\int_{\C^n}f(u)e^{\alpha\langle z,u\rangle}d\lambda_{\alpha}(u)\right|
    e^{-\frac{\alpha}{2}|z|^2}w(z)dv(z)\\
&\lesssim\sum_{\nu\in r\Z^{2n}}w(Q_r(\nu))\sum_{\nu'\in r\Z^{2n}}e^{-\frac{\alpha}{4}|\nu-\nu'|^2}\int_{Q_r(\nu')}|\tilde{f}|dv\\
&\leq\sum_{\nu'\in r\Z^{2n}}\sum_{\nu\in r\Z^{2n}}e^{-\frac{\alpha}{4}|\nu-\nu'|^2}C^{|\nu-\nu'|}w(Q_r(\nu'))\int_{Q_r(\nu')}|\tilde{f}|dv\\
&\lesssim\sum_{\nu'\in r\Z^{2n}}w(Q_r(\nu'))\|\sigma^{-1}\|_{L^{\infty}(Q_r(\nu'))}\int_{Q_r(\nu')}|\tilde{f}|\sigma dv\\
&\lesssim[w,\sigma]_{A_{1,r}}\|f\|_{L^1_{\alpha,\sigma}},
\end{align*}
which finishes the proof.
\end{proof}

\begin{remark}
By the proof of Theorem \ref{main3}, we know that the necessary part holds for any two weights $w$ and $\sigma$, i.e., if $P_{\alpha}:L^p_{\alpha,\sigma}\to L^p_{\alpha,w}$ is bounded, then $[w,\sigma]_{A_{p,r}}<\infty$ for any $r>0$.
\end{remark}

\section{Weighted compactness on Bergman spaces}\label{Bergman}

In this section, we consider the weighted compactness of Toeplitz operators on the Bergman spaces over the unit ball $\mathbb{B}_n$ of $\C^n$. Given $1<p<\infty$ and a weight $\sigma$ on $\bb_n$, the weighted space $L^p_{\sigma}$ consists of measurable functions $f$ on $\bb_n$ such that
$$\|f\|^p_{L^p_{\sigma}}:=\int_{\bb_n}|f|^p\sigma dv.$$
Let $\mathcal{H}(\bb_n)$ be the space of holomorphic functions on $\bb_n$. The weighted Bergman space $\A^p_{\sigma}$ is defined by
$$\A^p_{\sigma}:=L^p_{\sigma}\cap\mathcal{H}(\bb_n)$$
with the inherited norm. If $\sigma\equiv v(\bb_n)^{-1}$, then we obtain the Lebesgue space $L^p$ and the unweighted Bergman space $\A^p$.

It is well-known that $\A^2$ is a closed subspace of $L^2$, so there exists an orthogonal projection from $L^2$ onto $\A^2$. This map is called the Bergman projection and is given by
$$Pf(z):=\int_{\bb_n}f(w)\overline{K_z(w)}dv(w),$$
where $K_z(w)=\frac{1}{(1-\langle w,z\rangle)^{n+1}}$ is the reproducing kernel of $\A^2$.

B\'{e}koll\`{e} and Bonami \cite{Be,BB} initiated the study of the weighted theory of the Bergman projection. Given $a\in\bb_n\setminus\{0\}$, let $\ct_a$ be the Carleson tent defined by
$$\ct_a:=\left\{z\in\bb_n:|1-\langle z,a/|a|\rangle|<1-|a|\right\},$$
and we write $\ct_0=\bb_n$. A weight $\sigma$ on $\bb_n$ is said to be a B\'{e}koll\`{e}--Bonami $B_p$-weight, denoted by $\sigma\in B_p$, if
$$[\sigma]_{B_p}:=\sup_{a\in\bb_n}\left(\frac{1}{v(\ct_a)}\int_{\ct_a}\sigma dv\right)
  \left(\frac{1}{v(\ct_a)}\int_{\ct_a}\sigma^{-\frac{p'}{p}}dv\right)^{\frac{p}{p'}}<\infty.$$
B\'{e}koll\`{e} and Bonami \cite{Be,BB} showed that for $1<p<\infty$, the Bergman projection $P$ is bounded on $L^p_{\sigma}$ if and only if $\sigma\in B_p$.

Recently, Stockdale and Wagner \cite{SW} studied the weighted theory of Bergman--Toeplitz operators. Given $\varphi\in L^{\infty}(\bb_n,dv)$, the Bergman--Toeplitz operator $T_{\varphi}$ is defined by $T_{\varphi}f:=P(\varphi f)$. Stockdale and Wagner characterized the compactness of $T_{\varphi}$ on the weighted Bergman spaces $\A^p_{\sigma}$ for $\sigma$ in a subclass of $B_p$. Given $r>1$, we say that a weight $\sigma$ is in the reverse H\"{o}lder class $\mathrm{RH}_r$ if there exists $C>0$ such that for any $K\in\mathcal{D}$,
$$\left(\frac{1}{v(\widehat{K})}\int_{\widehat{K}}\sigma^rdv\right)^{\frac{1}{r}}\leq \frac{C}{v(\widehat{K})}\int_{\widehat{K}}\sigma dv.$$
Here, $\mathcal{D}$ is a fixed finite collection of dyadic systems of $\bb_n$, whose elements are ``dyadic cubes'' in $\bb_n$, and for $K\in \mathcal{D}$, $\widehat{K}$ denotes the associated dyadic tents; see \cite{SW} for the detials. Stockdale and Wagner \cite[Theorem 1.3]{SW} established the following result. Recall that the Berezin transform $\widetilde{T}_{\varphi}$ is defined by
$$\widetilde{T}_{\varphi}(z):=\langle T_\varphi k_z,k_z\rangle_{\A^2},$$
where $k_z=K_z/\|K_z\|_{\A^2}$ is the normalized reproducing kernel for $\A^2$.

\begin{theory}[\cite{SW}]\label{rh}{\it
Let $\varphi\in L^{\infty}(\bb_n,dv)$, $1<p<\infty$, $r>1$ and $\sigma\in B_p\cap\text{RH}_r$ with $\sigma^{1-p'}\in\text{RH}_r$. Then $T_\varphi$ acts compactly on $\A^p_{\sigma}$ if and only if $\widetilde{T}_{\varphi}(z)\to0$ as $|z|\to1^-$.}
\end{theory}

The main result of this section is as follows, which removes the reverse H\"{o}lder condition in Theorem \ref{rh} and answers \cite[Open Question 1.14]{SW}.

\begin{theorem}\label{main4}
Let $\varphi\in L^{\infty}(\bb_n,dv)$, $1<p<\infty$ and $\sigma\in B_p$. Then $T_{\varphi}$ acts compactly on $\A^p_{\sigma}$ if and only if $\widetilde{T}_{\varphi}(z)\to0$ as $|z|\to1^-$.
\end{theorem}

To prove the above theorem, we need some preliminary results. Recall that the Bergman metric $\beta(\cdot,\cdot)$ on $\bb_n$ is defined by
$$\beta(z,u):=\frac{1}{2}\log\left(\frac{1+|\phi_z(u)|}{1-|\phi_z(u)|}\right),\quad z,u\in\bb_n,$$
where $\phi_z$ is the involution of $\bb_n$ interchanging $z$ and $0$. Given $z\in\bb_n$ and $\delta>0$, we use $D(z,\delta)$ to denote the Bergman metric ball centered at $z$ with radius $\delta$. For a weight $\sigma$ on $\bb_n$, define
$$\hat{\sigma}(z):=\frac{\sigma(D(z,1))}{v(D(z,1))},\quad z\in\bb_n.$$
The following lemma indicates that $\hat{\sigma}\in B_p$ whenever $\sigma\in B_p$.

\begin{lemma}\label{hat}
Let $1<p<\infty$ and $\sigma\in B_p$. Then $\hat{\sigma}\in B_p$, and
$$[\hat{\sigma}]_{B_p}\lesssim[\sigma]_{B_p},$$
where the implicit constant depends only on $n$.
\end{lemma}
\begin{proof}
Recall that for $a\in\bb_n$, $v(\mathcal{T}_a)\asymp(1-|a|^2)^{n+1}$. Fix $a\in\bb_n$. If $|a|<\frac{19}{20}$, then by Fubini's theorem,
$$\frac{1}{v(\ct_a)}\int_{\ct_a}\hat{\sigma}(z)dv(z)
\lesssim\int_{\bb_n}\frac{\sigma(D(z,1))}{(1-|z|^2)^{n+1}}dv(z)\asymp\int_{\bb_n}\sigma dv,$$
and similarly, by H\"{o}lder's inequality,
$$\frac{1}{v(\ct_a)}\int_{\ct_a}\hat{\sigma}(z)^{-\frac{p'}{p}}dv(z)
\leq\frac{1}{v(\ct_a)}\int_{\ct_a}\frac{\int_{D(z,1)}\sigma^{-\frac{p'}{p}}dv}{v(D(z,1))}dv(z)
  \lesssim\int_{\bb_n}\sigma^{-\frac{p'}{p}}dv.$$
Consequently,
$$\left(\frac{1}{v(\ct_a)}\int_{\ct_a}\hat{\sigma}dv\right)
    \left(\frac{1}{v(\ct_a)}\int_{\ct_a}\hat{\sigma}^{-\frac{p'}{p}}dv\right)^{\frac{p}{p'}}
\lesssim\left(\int_{\bb_n}\sigma dv\right)\left(\int_{\bb_n}\sigma^{-\frac{p'}{p}}dv\right)^{\frac{p}{p'}}\lesssim[\sigma]_{B_p}.$$
Suppose now that $|a|>\frac{19}{20}$ and write $\tilde{a}=(1-20(1-|a|))\frac{a}{|a|}$. Then $v(\ct_a)\asymp v(\ct_{\tilde{a}})$, and we claim that
\begin{equation}\label{contain}
\bigcup_{z\in\ct_a}D(z,1)\subset\ct_{\tilde{a}}.
\end{equation}
In fact, for any $z\in\ct_a$ and $u\in D(z,1)$, we have
\begin{align*}
\left|1-\left\langle u,\tilde{a}/|\tilde{a}|\right\rangle\right|^{1/2}
&\leq|1-\langle u,z\rangle|^{1/2}+\left|1-\left\langle z,a/|a|\right\rangle\right|^{1/2}\\
&\leq|1-\langle\phi_z(\phi_z(u)),\phi_z(0)\rangle|^{1/2}+(1-|a|)^{1/2}\\
&=\frac{(1-|z|^2)^{1/2}}{|1-\langle\phi_z(u),z\rangle|^{1/2}}+(1-|a|)^{1/2}\\
&\leq\frac{\sqrt{2}(1-|a|)^{1/2}}{(1-|\phi_z(u)|)^{1/2}}+(1-|a|)^{1/2}\\
&<\left(\sqrt{\frac{2}{1-\tanh1}}+1\right)(1-|a|)^{1/2}.
\end{align*}
Note that $\tanh1\approx0.76<4/5$, we obtain that
$$\left|1-\left\langle u,\tilde{a}/|\tilde{a}|\right\rangle\right|<(\sqrt{10}+1)^2(1-|a|)<20(1-|a|)=1-|\tilde{a}|,$$
which gives that $u\in\ct_{\tilde{a}}$. Hence \eqref{contain} holds. Then by Fubini's theorem and \eqref{contain},
$$\frac{1}{v(\ct_a)}\int_{\ct_a}\hat{\sigma}(z)dv(z)
\asymp\frac{1}{v(\ct_a)}\int_{\ct_a}\frac{\sigma(D(z,1))}{(1-|z|^2)^{n+1}}dv(z)
\lesssim\frac{1}{v(\ct_{\tilde{a}})}\int_{\ct_{\tilde{a}}}\sigma dv$$
and
$$\frac{1}{v(\ct_a)}\int_{\ct_a}\hat{\sigma}(z)^{-\frac{p'}{p}}dv(z)
\leq\frac{1}{v(\ct_a)}\int_{\ct_a}\frac{\int_{D(z,1)}\sigma^{-\frac{p'}{p}}dv}{v(D(z,1))}dv(z)
\lesssim\frac{1}{v(\ct_{\tilde{a}})}\int_{\ct_{\tilde{a}}}\sigma^{-\frac{p'}{p}}dv.$$
Therefore,
\begin{align*}
\left(\frac{1}{v(\ct_a)}\int_{\ct_a}\hat{\sigma}dv\right)
    &\left(\frac{1}{v(\ct_a)}\int_{\ct_a}\hat{\sigma}^{-\frac{p'}{p}}dv\right)^{\frac{p}{p'}}\\
&\lesssim\left(\frac{1}{v(\ct_{\tilde{a}})}\int_{\ct_{\tilde{a}}}\sigma dv\right)
\left(\frac{1}{v(\ct_{\tilde{a}})}\int_{\ct_{\tilde{a}}}\sigma^{-\frac{p'}{p}}dv\right)^{\frac{p}{p'}}\leq[\sigma]_{B_p}.
\end{align*}
Since $a\in\bb_n$ is arbitrary, we conclude that $\hat{\sigma}\in B_p$, and $[\hat{\sigma}]_{B_p}\lesssim[\sigma]_{B_p}$.
\end{proof}

We also need the notion of $C_p$-weights. For $1<p<\infty$, a weight $\sigma$ on $\bb_n$ is said to be in the $C_p$-class if
$$\sup_{a\in\bb_n}\left(\frac{1}{v(D(a,1))}\int_{D(a,1)}\sigma dv\right)
\left(\frac{1}{v(D(a,1))}\int_{D(a,1)}\sigma^{-\frac{p'}{p}}dv\right)^{\frac{p}{p'}}<\infty.$$
It was pointed out in \cite[p. 322]{Lu} that $B_p\subset C_p$. Moreover, by \cite[Corollary 3.8]{Lu}, if $\sigma\in C_p$ for some $1<p<\infty$, then for all $R>0$,
\begin{equation}\label{tilde}
\hat{\sigma}(z)\asymp\hat{\sigma}(u)\quad \mathrm{whenever} \quad \beta(z,u)<R.
\end{equation}
A direct consequence of this fact is that $\hat{\sigma}\in C_p$. The following lemma gives an equivalent norm for the weighted Bergman spaces $\A^p_{\sigma}$ with $\sigma\in C_p$.

\begin{lemma}\label{bnormeq}
Let $1<p<\infty$ and $\sigma\in C_p$. Then $\A^p_{\sigma}=\A^p_{\hat{\sigma}}$ with equivalent norms.
\end{lemma}
\begin{proof}
We only prove the bounded embedding $\A^p_{\hat{\sigma}}\subset \A^p_{\sigma}$. The other is similar. Suppose that $f\in \A^p_{\hat{\sigma}}$. Let $0<\delta<1$ and $\{a_k\}\subset\bb_n$ be a $\delta$-lattice in the Bergman metric. Then $\bb_n=\bigcup_kD(a_k,\delta)$ and there exists an absolute constant $N$ such that each point $z\in\bb_n$ belongs to at most $N$ of the sets $D(a_k,4\delta)$ (see \cite[Theorem 2.23]{Zhuball}). Consequently, by \cite[Lemma 3.1]{Lu} and \eqref{tilde},
\begin{align*}
\int_{\bb_n}|f(z)|^p\sigma(z)dv(z)
&\leq\sum_{k}\int_{D(a_k,\delta)}|f(z)|^p\sigma(z)dv(z)\\
&\lesssim\sum_{k}\int_{D(a_k,\delta)}\frac{1}{\hat{\sigma}(z)(1-|z|^2)^{n+1}}\int_{D(z,\delta)}|f|^p\hat{\sigma}dv\sigma(z)dv(z)\\
&\leq\sum_k\int_{D(a_k,\delta)}\frac{\sigma(z)dv(z)}{\hat{\sigma}(z)(1-|z|^2)^{n+1}}\int_{D(a_k,2\delta)}|f|^p\hat{\sigma}dv\\
&\lesssim\sum_k\int_{D(a_k,2\delta)}|f|^p\hat{\sigma}dv\\
&\lesssim\|f\|^p_{\A^p_{\hat{\sigma}}}.
\end{align*}
Hence $f\in \A^p_{\sigma}$ and $\|f\|_{\A^p_{\sigma}}\lesssim\|f\|_{\A^p_{\hat{\sigma}}}$, which finishes the proof.
\end{proof}

We are now ready to prove Theorem \ref{main4}.

\begin{proof}[Proof of Theorem \ref{main4}]
Since $\sigma\in B_p\subset C_p$, it follows from Lemma \ref{bnormeq} that $T_{\varphi}$ is compact on $\A^p_{\sigma}$ if and only if it is compact on $\A^p_{\hat{\sigma}}$. Therefore, in view of Theorem \ref{rh}, it suffices to verify that there exists $r>1$ such that $\hat{\sigma}\in B_p\cap\mathrm{RH}_r$ and $\hat{\sigma}^{1-p'}\in\mathrm{RH}_r$. By Lemma \ref{hat}, $\hat{\sigma}\in B_p$, and consequently, $\hat{\sigma}^{1-p'}\in B_{p'}$, which, combined with \eqref{tilde} and \cite[Theorem 2.9]{SW}, shows that $\hat{\sigma},\hat{\sigma}^{1-p'}\in\mathrm{RH}_r$ for some $r>1$. The proof is complete.
\end{proof}

\medskip





\begin{thebibliography}{99}
	
\bibitem{APR} A. Aleman, S. Pott and M. C. Reguera,
\newblock{Sarason conjecture on the Bergman space,}
\newblock Int. Math. Res. Not. IMRN 2017, no. 14, 4320--4349.
	
\bibitem{AZ} S. Axler and D. Zheng,
\newblock{Compact operators via the Berezin transform,}
\newblock Indiana Univ. Math. J. 47 (1998), no. 2, 387--400.
	

\bibitem{BCK} W. Bauer, B. R. Choe and H. Koo,
\newblock{Commuting Toeplitz operators with pluriharmonic symbols on the Fock space,}
\newblock J. Funct. Anal. 268 (2015), no. 10, 3017--3060.
	
\bibitem{BCI} W. Bauer, L. A. Coburn and J. Isralowitz,
\newblock{Heat flow, BMO, and the compactness of Toeplitz operators,}
\newblock J. Funct. Anal. 259 (2010), no. 1, 57--78.

\bibitem{BI} W. Bauer and J. Isralowitz,
\newblock{Compactness characterization of operators in the Toeplitz algebra of the Fock space $F^p_{\alpha}$,}
\newblock J. Funct. Anal. 263 (2012), no. 5, 1323--1355.

\bibitem{BvSW} W. Bauer, L. van Luijk, A. Stottmeister and R. F. Werner,
\newblock{Self-adjointness of Toeplitz operators on the Segal--Bargmann space,}
\newblock J. Funct. Anal. 284 (2023), no. 4, Paper No. 109778, 24 pp.

\bibitem{Be} D. B\'{e}koll\`{e},
\newblock{In\'{e}galit\'{e} \`{a} poids pour le projecteur de Bergman dans la boule unit\'{e} de $\mathbf{C}^n$,}
\newblock Studia Math. 71 (1981/82), no. 3, 305--323.

\bibitem{BB} D. B\'{e}koll\`{e} and A. Bonami,
\newblock{In\'{e}galit\'{e}s \`{a} poids pour le noyau de Bergman,}
\newblock C. R. Acad. Sci. Paris S\'{e}r. A-B 286 (1978), no. 18, A775--A778.

\bibitem{BEY} H. Bommier-Hato, M. Engli\v{s} and E.-H. Youssfi,
\newblock{Bergman-type projections in generalized Fock spaces,}
\newblock J. Math. Anal. Appl. 389 (2012), no. 2, 1086--1104.

\bibitem{Br} G. Brocchi,
\newblock{Refined two weight estimates for the Bergman projection,}
\newblock Collect. Math. 76 (2025), no. 1, 65--80.

\bibitem{CFP} C. Cascante, J. F\`{a}brega and J. \'{A}. Pel\'{a}ez,
\newblock{Littlewood--Paley formulas and Carleson measures for weighted Fock spaces induced by $A_{\infty}$-type weights,}
\newblock Potential Anal. 50 (2019), no. 2, 221--244.

\bibitem{Ch24} J. Chen,
\newblock{Composition operators on weighted Fock spaces induced by $A_{\infty}$-type weights,}
\newblock Ann. Funct. Anal. 15 (2024), no. 2, Paper No. 22, 24 pp.

\bibitem{CHW} J. Chen, B. He and M. Wang,
\newblock{Absolutely summing Carleson embeddings on weighted Fock spaces with $A_{\infty}$-type weights,}
\newblock J. Operator Theory, in press.

\bibitem{CW24} J. Chen and M. Wang,
\newblock{Weighted norm inequalities, embedding theorems and integration operators on vector-valued Fock spaces,}
\newblock Math. Z. 307 (2024), no. 2, Paper No. 36, 30 pp.

\bibitem{CW24-1} J. Chen and M. Wang,
\newblock{Fock projections on vector-valued $L^p$-spaces with matrix weights,}
\newblock preprint, 2024. arXiv:2408.13537

\bibitem{CIL} L. A. Coburn, J. Isralowitz and B. Li,
\newblock{Toeplitz operators with BMO symbols on the Segal--Bargmann space,}
\newblock Trans. Amer. Math. Soc. 363 (2011), no. 6, 3015--3030.



\bibitem{En} M. Engli\v{s},
\newblock{Compact Toeplitz operators via the Berezin transform on bounded symmetric domains,}
\newblock Integral Equations Operator Theory 33 (1999), no. 4, 426--455.

\bibitem{FW} X. Fang and Z. Wang,
\newblock{Two weight inequalities for the Bergman projection with doubling measures,}
\newblock Taiwanese J. Math. 19 (2015), no. 3, 919--926.

\bibitem{Fu} R. Fulsche,
\newblock{Essential positivity for Toeplitz operators on the Fock space,}
\newblock Integral Equations Operator Theory 96 (2024), no. 3, Paper No. 21, 10 pp.

\bibitem{Fu2} R. Fulsche,
\newblock{Toeplitz operators on non-reflexive Fock spaces,}
\newblock Rev. Mat. Iberoam. 40 (2024), no. 3, 1115--1148.

\bibitem{HL} Z. Hu and X. Lv,
\newblock{Toeplitz operators from one Fock space to another,}
\newblock Integral Equations Operator Theory 70 (2011), no. 4, 541--559.

\bibitem{HLW} Z. Hu, X. Lv and B. D. Wick,
\newblock{Localization and compactness of operators on Fock spaces,}
\newblock J. Math. Anal. Appl. 461 (2018), no. 2, 1711--1732.

\bibitem{Is14} J. Isralowitz,
\newblock{Invertible Toeplitz products, weighted norm inequalities, and $A_p$ weights,}
\newblock J. Operator Theory 71 (2014), no. 2, 381--410.

\bibitem{Is15} J. Isralowitz,
\newblock{Compactness and essential norm properties of operators on generalized Fock spaces,}
\newblock J. Operator Theory 73 (2015), no. 2, 281--314.

\bibitem{IMW} J. Isralowitz, M. Mitkovski and B. D. Wick,
\newblock{Localization and compactness in Bergman and Fock spaces,}
\newblock Indiana Univ. Math. J. 64 (2015), no. 5, 1553--1573.

\bibitem{IZ} J. Isralowitz and K. Zhu,
\newblock{Toeplitz operators on the Fock space,}
\newblock Integral Equations Operator Theory 66 (2010), no. 4, 593--611.

\bibitem{JPR} S. Janson, J. Peetre and R. Rochberg,
\newblock{Hankel forms and the Fock space,}
\newblock Rev. Mat. Iberoamericana 3 (1987), no. 1, 61--138.

\bibitem{JZ} M. Jovovic and D. Zheng,
\newblock{Compact operators and Toeplitz algebras on multiply-connected domains,}
\newblock J. Funct. Anal. 261 (2011), no. 1, 25--50.

\bibitem{La14} M. T. Lacey,
\newblock{Two-weight inequality for the Hilbert transform: a real variable characterization II,}
\newblock Duke Math. J. 163 (2014), no. 15, 2821--2840.

\bibitem{LSSU} M. T. Lacey, E. T. Sawyer, C.-Y. Shen and I. Uriarte-Tuero,
\newblock{Two-weight inequality for the Hilbert transform: a real variable characterization I,}
\newblock Duke Math. J. 163 (2014), no. 15, 2795--2820.

\bibitem{LSU} M. T. Lacey, E. T. Sawyer and I. Uriarte-Tuero,
\newblock{A two weight inequality for the Hilbert transform assuming an energy hypothesis,}
\newblock J. Funct. Anal. 263 (2012), no. 2, 305--363.

\bibitem{Lu} D. H. Luecking,
\newblock{Representation and duality in weighted spaces of analytic functions,}
\newblock Indiana Univ. Math. J. 34 (1985), no. 2, 319--336.


\bibitem{MSWW} M. Mitkovski, C. B. Stockdale, N. A. Wagner and B. D. Wick,
\newblock{Riesz--Kolmogorov type compactness criteria in function spaces with applications,}
\newblock Complex Anal. Oper. Theory 17 (2023), no. 3, Paper No. 40, 31 pp.

\bibitem{MSW} M. Mitkovski, D. Su\'{a}rez and B. D. Wick,
\newblock{The essential norm of operators on $A^p_{\alpha}(\mathbb{B}_n)$,}
\newblock Integral Equations Operator Theory 75 (2013), no. 2, 197--233.

\bibitem{Sa82} E. T. Sawyer,
\newblock{A characterization of a two-weight norm inequality for maximal operators,}
\newblock Studia Math. 75 (1982), no. 1, 1--11.

\bibitem{Sa88} E. T. Sawyer,
\newblock{A characterization of two weight norm inequalities for fractional and Poisson integrals,}
\newblock Trans. Amer. Math. Soc. 308 (1988), no. 2, 533--545.

\bibitem{SW} C. B. Stockdale and N. A. Wagner,
\newblock{Weighted theory of Toeplitz operators on the Bergman space,}
\newblock Math. Z. 305 (2023), no. 1, Paper No. 10, 29 pp.

\bibitem{Su} D. Su\'{a}rez,
\newblock{The essential norm of operators in the Toeplitz algebra on $A^p(\mathbb{B}_n)$,}
\newblock Indiana Univ. Math. J. 56 (2007), no. 5, 2185--2232.

\bibitem{WX} Y. Wang and J. Xia,
\newblock{Essential commutants on strongly pseudo-convex domains,}
\newblock J. Funct. Anal. 280 (2021), no. 1, Paper No. 108775, 56 pp.


\bibitem{XZ} J. Xia and D. Zheng,
\newblock{Localization and Berezin transform on the Fock space,}
\newblock J. Funct. Anal. 264 (2013), no. 1, 97--117.

\bibitem{Zhuball} K. Zhu,
\newblock{Spaces of holomorphic functions in the unit ball,}
\newblock Graduate Texts in Mathematics, 226. Springer-Verlag, New York, 2005.

\bibitem{Zh} K. Zhu,
\newblock{Analysis on Fock spaces,}
\newblock Graduate Texts in Mathematics, 263. Springer, New York, 2012.






\end{thebibliography}
\end{document}